\documentclass[12pt,a4paper]{amsart}
\usepackage{amsmath,amssymb,amsthm}
\usepackage{hyperref}
\usepackage{enumitem}
\usepackage{mathrsfs}
\usepackage{tikz-cd}
\usepackage{amscd}
\usepackage[only,llbracket,rrbracket]{stmaryrd}

\usepackage{geometry}
\long\def\forget#1{}

\usepackage{amsthm}
\theoremstyle{plain}
\newtheorem{theorem}{Theorem}[section]

\newtheorem{corollary}[theorem]{Corollary}
\newtheorem{proposition}[theorem]{Proposition}

\theoremstyle{definition}

\newtheorem{definition-theorem}[theorem]{Definition-Theorem}
\newtheorem{definition-remark}[theorem]{Definition-Remark}

\newtheorem{remark}[theorem]{Remark}

\theoremstyle{remark}

\newcommand{\myparagraph}[1]{}

\newcommand{\opCH}{\mathrm{CH}_\mathrm{op}}
\newcommand{\CHM}{\mathrm{CHM}}
\newcommand{\DM}{\mathbf{DM}}
\newcommand{\Mcomp}{M^c}
\newcommand{\LL}{\mathbb L}
\newcommand{\one}{\mathbf 1}
\newcommand{\Spec}{\operatorname{Spec}}
\newcommand{\Hom}{\operatorname{Hom}}

\newcommand{\HH}{\mathrm H}

\title{On The Motivic Leray-Hirsch Theorem For Pure Tate Fibre Bundles}
\author{Esmail Arasteh Rad and Somayeh Habibi}

\date{\today}

\begin{document}

\begin{abstract}
We prove a Leray--Hirsch theorem for a smooth proper morphism over a smooth integral base whose generic fibre has a pure Tate motive. We further explore and highlight some key applications of this theorem.
\end{abstract}

\maketitle



\section*{Introduction}\label{SectIntroduction}

The classical Leray–Hirsch theorem describes the cohomology of a fibre bundle in terms of the cohomology of its base and fibre, provided that a basis of the cohomology of the fibre admits global lifts to the total space. These lifts give an explicit module decomposition and supply a splitting of the filtration arising from the Leray–Serre spectral sequence. An algebraic analogue asks whether algebraic cycles on the fibres of a smooth proper morphism can similarly be used to decompose the Chow groups, and ultimately the motive, of its total space.

In algebraic geometry, this question requires additional hypotheses on the fibre. Chow groups do not satisfy a general Künneth formula, even for products of smooth projective varieties. Cellular varieties provide an important setting in which this obstruction disappears. A Leray–Hirsch theorem for Chow groups of Zariski locally trivial fibrations with cellular fibre was established by Collino and Fulton \cite{CF}. Motivic decompositions for cellular fibrations, together with applications to computing motive of certain $G$-bundles and towers of fibrations, were studied in our earlier work \cite{AH17}. These results suggest that the essential hypothesis should be formulated in terms of the motive of the fibre rather than a chosen cellular decomposition. The purpose of this paper is to establish a Leray–Hirsch theorem under a hypothesis on the motive of the generic fibre. Let \(k\) be a perfect field, let \(Y\) be a smooth integral \(k\)-variety with function field \(K=k(Y)\), and let \(\pi\colon X\to Y\) be a smooth projective morphism whose generic fibre \(F=X_K\) has an integral pure Tate motive. \forget{
\[
M_{\mathrm{CH}}(F)
\simeq
\bigoplus_{j=0}^{d}
\left(\mathbb L_K^{\otimes j}\right)^{\oplus r_j}.
\]
The decomposition is required over \(K\) itself. 
} Our main theorem provides an explicit decomposition of the Chow groups
of the total space (without requiring $\pi$ to be Zariski-locally
trivial or its fibres to admit cellular decompositions). Operational Chow cohomology provides the appropriate
language for restricting the chosen global cycle classes to possibly singular closed subschemes of the base and for
defining their action on homological Chow groups, see Section \ref{SectNotation} and proof of Theorem \ref{ThmLetaHirschForChow}. 

When \(Y\) is projective, the Chow-group decomposition admits a refinement in the category of Chow motives with integral coefficients; See Theorem \ref{thm:chowmotive}.  

For \(Y\) not necessarily proper, we also consider a relative formulation in the category
\(\mathbf{DM}(Y,R)\), where \(R=\mathbb Z[1/p]\) and \(p\) is the exponential characteristic of \(k\). In this setting the decomposition takes the form
\[
\pi_*\mathbf1_X
\simeq
\bigoplus_{j=0}^{d}
\mathbf1_Y(-j)[-2j]^{\oplus r_j},
\]
which gives the corresponding decomposition of motives with compact support, see Theorem \ref{thm:motivic}, and hence analogous formulas for higher Chow groups with coefficients in \(R\); See Corollary \ref{CorHigherChow}. 
We conclude with briefly mentioning several applications,  including projective homogeneous bundles associated with a $G$-bundle. Iterating the theorem also allows us to
compute the motives of towers of such fibrations, which together with motivic \emph{decomposition theorem} \cite{C-M}, has  applications to study the motive associated with Schubert varieties in affine Grassmannian via the Bott-Samelson-Demazure type resolutions; See \cite{AH-Shtukas}.


\section{Notation and Preliminaries}\label{SectNotation}

Throughout, $k$ is a perfect field, and all $k$-schemes are
assumed to be separated and of finite type over $k$. We denote the
exponential characteristic of $k$ by $p$; thus $p=1$ in
characteristic zero and $p=\operatorname{char}(k)$ otherwise.
Chow groups have integral coefficients unless explicitly
indicated. In statements involving the six-functor formalism
or motives with compact support over arbitrary schemes, we
work with coefficients in $R\subseteq \mathbb Q$ in
which the characteristic exponent of $k$ is invertible. .

\medskip
\noindent
\textbf{Chow groups and operational Chow cohomology.}
For a $k$-scheme $X$, let $\mathrm{CH}_m(X)$ denote the
group of algebraic cycles of dimension $m$ modulo rational
equivalence. If $X$ is equidimensional of dimension $d$, we
write
\[
\mathrm{CH}^j(X)=\mathrm{CH}_{d-j}(X).
\]
For smooth $X$, these groups carry the usual intersection
product. 

For an arbitrary scheme $X$, we write $\opCH^j(X)$ for operational
Chow cohomology with integral coefficients; see
\cite[Definition 17.3]{Ful}. A class $c\in\opCH^j(X)$
assigns to every morphism $g\colon T\to X$ homomorphisms
\[
c_g\colon\mathrm{CH}_m(T)\longrightarrow
\mathrm{CH}_{m-j}(T)
\]
for all $m$, compatible with proper pushforward, flat
pullback, and refined Gysin homomorphisms for regular
embeddings.\forget{These compatibilities are required in every
Cartesian diagram in which the corresponding operations
are defined.} We also write
\[
c_g(\alpha)=g^*c\cap\alpha.
\]
Operational classes pull back along arbitrary morphisms,
and composition of their actions defines the operational
Chow ring $\opCH^*(X)$.

If $X$ is smooth and equidimensional of dimension $d$,
evaluation on its fundamental class induces a canonical
isomorphism
\[
\opCH^j(X)\xrightarrow{\ \sim\ }
\mathrm{CH}_{d-j}(X)=\mathrm{CH}^j(X),
\qquad
c\longmapsto c_{\mathrm{id}_X}([X]).
\]
Under this identification, composition of operational
classes agrees with the intersection product; see
\cite[Corollary 17.4]{Ful} and
\cite[Tag 0FC2]{Stacks}. See also \cite[Section 42.33]{Stacks}.
In particular, an ordinary Chow class on a smooth scheme
can be pulled back to a possibly singular
scheme and then applied to its homological Chow groups.

\medskip
\noindent
\textbf{Higher Chow groups.}
Let $z_m(X,*)$ denote Bloch's dimension-indexed cycle complex; See \cite{Bloch}. Recall that $z_m(X,n)$ is generated by integral
closed subschemes of $X\times\Delta^n$ of dimension $m+n$
meeting all faces properly, where
\[
\Delta^n=
\operatorname{Spec}
k[t_0,\ldots,t_n]/(t_0+\cdots+t_n-1),
\]
and the differentials are the alternating sum of the face maps.
We write
\[
\mathrm{CH}_m(X,n)=H_n(z_m(X,*)),
\qquad n\geq0.
\]
If $X$ is equidimensional of dimension $d$, the
codimension-indexed convention is
\[
\mathrm{CH}^j(X,n)=H_n(z^j(X,*)),
\qquad
\mathrm{CH}_m(X,n)=\mathrm{CH}^{d-m}(X,n).
\]
Note that \(\mathrm{CH}_m(X,0)=\mathrm{CH}_m(X).\) With coefficients in $R$, we set
\[
\mathrm{CH}_m(X,n;R)
=
\mathrm{CH}_m(X,n)\otimes_{\mathbb Z}R.
\]

\textbf{Chow motives.}
We use the covariant convention with coefficients in $R$.
For smooth projective $k$-schemes $X$ and $Y$, with $X$
connected of dimension $d_X$, set
\[
 \operatorname{Corr}_r(X,Y;R)
 :=\mathrm{CH}_{d_X+r}(X\times_k Y;R).
\]
For disconnected $X$, this definition is applied
componentwise. Composition is given by pull-back,
intersection, and proper push-forward along the
projections from $X\times_k Y\times_k Z$.

The category $\mathrm{CHM}(k;R)$ has objects $(X,e,m)$,
where $X$ is smooth projective over $k$,
$e\in\operatorname{Corr}_0(X,X;R)$ is idempotent,
and $m\in\mathbb Z$, with morphisms
\[
 \operatorname{Hom}_{\mathrm{CHM}(k;R)}
 \bigl((X,e,m),(Y,f,n)\bigr)
 =
 f\circ\operatorname{Corr}_{n-m}(X,Y;R)\circ e.
\]
If $Y$ is equidimensional of dimension $d_Y$, then
\[
 \operatorname{Corr}_{n-m}(X,Y;R)
 =
 \mathrm{CH}^{d_Y+m-n}(X\times_k Y;R).
\]
This category is additive and idempotent complete.
Its tensor product is
\[
 (X,e,m)\otimes(Y,f,n)
 =
 (X\times_k Y,e\boxtimes f,m+n).
\]

We write
\[
 M_{\mathrm{CH}}(X)=(X,\Delta_X,0),
 \qquad
 \mathbf1(n)=(\Spec k,\mathrm{id},n),
 \qquad
 \mathbf1=\mathbf1(0).
\]
With this convention, the Lefschetz motive is
$\mathbb L=\mathbf1(-1)$, and
\[
 M_{\mathrm{CH}}(\mathbb P^1)
 \simeq\mathbf1\oplus\mathbb L.
\]
For every smooth projective $X$, there are canonical
identifications
\[
 \mathrm{CH}_j(X;R)
 \simeq
 \operatorname{Hom}_{\mathrm{CHM}(k;R)}
 \bigl(\mathbb L^{\otimes j},M_{\mathrm{CH}}(X)\bigr),
 \qquad
 \mathrm{CH}^j(X;R)
 \simeq
 \operatorname{Hom}_{\mathrm{CHM}(k;R)}
 \bigl(M_{\mathrm{CH}}(X),\mathbb L^{\otimes j}\bigr);
\]
See \cite[page 297]{EKM}. For integral coefficients, we abbreviate
$\mathrm{CHM}(k;\mathbb Z)$ to $\mathrm{CHM}(k)$.

\bigskip

A \emph{pure Tate Chow motive} is a finite direct sum
of tensor powers of $\mathbb L$. In particular, an effective pure Tate motive
has the form
\[
\bigoplus_{j\geq0}
\left(\mathbb L^{\otimes j}\right)^{\oplus r_j},
\]
where $r_j\geq0$ and only finitely many $r_j$ are nonzero. 

\begin{remark}
    Recall that a variety $F$ is called \emph{cellular} if there exists a filtration by closed subvarieties
    \[
    \emptyset = F_{-1} \subset F_0 \subset \dots \subset F_m = F
    \]
    such that $F_i \setminus F_{i-1}$ is a disjoint union of affine spaces $\coprod_j\mathbb A^{d_i^j}$. Then motive $M_{\mathrm{CH}}(F)$ associated to $F$ admits a decomposition
$$
M_{\mathrm{CH}}(F)\cong \bigoplus_{i,j} \mathbf 1(-d_i^j),
$$
for example see \cite{EKM}[Corollary 66.4]. In particular the motive associated to a cellular variety is pure Tate. 
\end{remark}

\medskip
\noindent
\textbf{Geometric motives and compact support.}
We follow \cite{VSF} for Voevodsky's categories
\[
\mathbf{DM}_{\mathrm{gm}}^{\mathrm{eff}}(k,R)
\quad\text{and}\quad
\mathbf{DM}_{\mathrm{gm}}(k,R);
\]
see also \cite{MVW}. We write $M(X;R)$ for the motive of
$X$ and $M^c(X;R)$ for its motive with compact support,
and omit $R$ when the coefficient ring is understood. Note that for arbitrary separated schemes of finite type, the
construction of motives with compact
support in \cite{VSF} use resolution of singularities over $k$.
Kelly \cite{Kelly} extends this construction to arbitrary perfect fields, provided that the exponential characteristic of the base field is invertible in \(R\).

The motive $M^c(X)$ is covariant for proper morphisms.
For a closed immersion $Z\hookrightarrow X$ with open
complement $U$, there is a localization triangle
\[
M^c(Z)\longrightarrow M^c(X)
\longrightarrow M^c(U)
\longrightarrow M^c(Z)[1].
\]
If $X$ is proper, the canonical comparison \(M(X)\xrightarrow{\ \sim\ }M^c(X) \) is an isomorphism.

\begin{remark}\label{RemEmbeddingCHtoDMgm}
With the covariant convention fixed above, correspondence
realization gives a fully faithful symmetric monoidal functor
\[
\mathrm{CHM}(k;R)
\longrightarrow
\mathbf{DM}_{\mathrm{gm}}(k,R)
\]
sending \(M_{\mathrm{CH}}(X)\longmapsto M(X)\) and \(\mathbb L\longmapsto R(1)[2]\). 
See \cite[Proposition~2.1.4]{VSF}.
\end{remark}

We call an object of $\mathbf{DM}_{\mathrm{gm}}(k,R)$
\emph{pure Tate} if it is a finite direct sum of objects
$R(j)[2j]$. 

\medskip
\noindent
\textbf{Relative motives and duality.}
For a $k$-scheme $S$, put
\[
\mathbf{DM}(S,R)=\mathbf{DM}_{\mathrm{cdh}}(S,R),
\]
and denote its tensor unit by $\mathbf1_S$. For the construction and the six-functor formalism we refer the reader to \cite{CD15}, compare also \cite{CD}. Over $k$, the
constructible objects identify with
$\mathbf{DM}_{\mathrm{gm}}(k,R)$.

For the structure morphism
$p\colon S\to\operatorname{Spec}k$, write \(D_S:=p^!R\), and \(\mathcal D_k(-):=
\underline{\operatorname{Hom}}(-,R)\) for the duality functor on $\mathbf{DM}_{\mathrm{gm}}(k,R)$.
There are canonical isomorphisms
\begin{equation}\label{eq:mcdual}
M^c(S)
\simeq p_*D_S
\simeq \mathcal D_k(p_!\mathbf1_S);
\end{equation}
see \cite[Proposition~8.10]{CD15}.

For an equidimensional quasi-projective $k$-scheme $S$ as before, consider any subring $R\subseteq \mathbb Q$ in
which the characteristic exponent of $k$ is invertible. Then the comparison with higher Chow groups takes the form

\begin{equation}\label{eq:higherhom}
\mathrm{CH}_m(S,n;R)
\simeq
\operatorname{Hom}_{\mathbf{DM}(S,R)}
\bigl(R_S(m)[2m+n],D_S\bigr),
\qquad n\geq0;
\end{equation}

\begin{equation}\label{eq:higherhom1}
\mathrm{CH}_m(S,n;R)
\simeq
\operatorname{Hom}_{\mathbf{DM}(k,R)}
\bigl(R(m)[2m+n],M^c(S;R)\bigr),
\qquad n\geq0;
\end{equation}
see \cite[Corollary 8.12]{CD15}. The proof is a slight $\mathbb Z[1/p]$-modification of \cite[Proposition~4.2.9]{VSF}. 
For $n=0$, the corresponding identification with
ordinary homological Chow groups holds for every
separated $k$-scheme of finite type; see
\cite[Corollary 8.13]{CD15}.

For a smooth $k$-scheme $X$, motivic cohomology also gives
\[
\mathrm{CH}^j(X;R)
\simeq
\operatorname{Hom}_{\mathbf{DM}_{\mathrm{gm}}(k,R)}
\bigl(M(X),R(j)[2j]\bigr);
\]
see \cite[Theorem~19.1 and Proposition 19.2]{MVW} and \cite
{VSF}.

\medskip
\noindent
\textbf{Pure Tate decompositions.}
Let $D^b_{\mathrm{fg}}(R)$ denote the bounded derived
category of finitely generated $R$-modules.
Since $R$ is a localization of $\mathbb Z$, every
bounded complex with finitely generated cohomology
is perfect. The constant-motive
construction gives a fully faithful symmetric monoidal
triangulated functor
\[
D^b_{\mathrm{fg}}(R)
\longrightarrow
\mathbf{DM}_{\mathrm{gm}}(k,R),
\qquad
R\longmapsto R(0).
\]
For the integral version, see
\cite[Proposition 4.5]{H-K}; the same construction applies
after localization of the coefficients.

In particular, if $A$ is a finite free $R$-module and
$N$ is a geometric motive, the expression $A\otimes_R N$
denotes the tensor product of $N$ with the constant motive
associated with $A$. Choosing a basis of $A$ identifies
this object with a finite direct sum of copies of $N$\forget{,
but its definition does not depend on that choice}.

The following proposition makes precise the reconstruction
of a pure Tate motive from its Chow groups; compare
\cite[Proposition 4.10]{H-K}.

\begin{proposition}\label{PropCHowDecompositionForPM}
Let $X$ be a smooth projective equidimensional $k$-scheme
of dimension $d$. Suppose that $M(X)$ is pure Tate.
Then the groups $\mathrm{CH}_j(X;R)$ and
$\mathrm{CH}^j(X;R)$ are finite free $R$-modules, and
there is a canonical isomorphism
\[
\bigoplus_{j=0}^{d}
\mathrm{CH}_j(X;R)\otimes_R R(j)[2j]
\xrightarrow{\ \sim\ }M(X).
\]
\forget{Its restriction to the summand corresponding to a class
$\alpha\in\mathrm{CH}_j(X;R)$ is the morphism
\[
R(j)[2j]\longrightarrow M(X)
\]
associated with $\alpha$ by the Chow-group comparison.}
\end{proposition}

\begin{proof}
Recall the natural Chow-group identification
\[
\mathrm{CH}_j(X;R)
 \simeq\Hom(R(j)[2j],M^c(X)).
\]
Note that since $M(X)$ is pure Tate both groups are finite free. By properness we have $M^c(X)\simeq M(X)$. This gives a canonical map
\[
 \bigoplus_{j=0}^d
 \mathrm{CH}_j(X;R)\otimes_R R(j)[2j]
 \longrightarrow M(X).
\]
This map induces isomorphisms
on $\Hom(R(i)[2i],-)$ for every $i$. Both sides are
pure Tate, so a Yoneda argument proves the assertion.
\end{proof}

\section{The Classical Leray--Hirsch Theorem and the
Leray--Serre Spectral Sequence}

Below we recall the classical Leray–Hirsch theorem. For the sake of completeness, and to make the analogy with the results developed in the sequel more transparent, we also provide a proof based on the Serre spectral sequence.

\begin{theorem}[Leray--Hirsch]\label{thm:leray-hirsch}
Let
\[
 F\xrightarrow{\iota}E\xrightarrow{\pi}B
\]
be a fibration, where $B$ and $F$ are path-connected finite
CW complexes. Let $R$ be a commutative ring with identity.
Assume that $\HH^*(F,R)$ is a finitely generated free graded
$R$-module and that the restriction homomorphism
\[
 \iota^*:\HH^*(E,R)\longrightarrow\HH^*(F,R)
\]
is surjective. Then there exists an isomorphism of graded
$\HH^*(B,R)$-modules
\[
 \HH^*(B,R)\otimes_R\HH^*(F,R)
 \xrightarrow{\ \sim\ }
 \HH^*(E,R),
\]
where $\HH^*(B,R)$ acts on the source by multiplication
on the first factor and on the target through $\pi^*$.
\end{theorem}

\begin{proof}
Consider the cohomological Leray--Serre spectral sequence
\[
 E_2^{p,q}
 =
 \HH^p\bigl(B,\mathcal H^q(F,R)\bigr)
 \Longrightarrow \HH^{p+q}(E,R),
\]
where $\mathcal H^q(F,R)$ denotes the local system of fibre
cohomology. The monodromy action of $\pi_1(B,b)$ on $\mathrm H^*(F_b,R)$ is trivial. This follows from the existence of global lifts, guaranteed by the surjectivity of $\iota^*$.

Choose a homogeneous $R$-basis $\{e_i\}_{i\in I}$
of $\HH^*(F,R)$, with $e_i\in\HH^{d_i}(F,R)$,
and put $I_q=\{i\in I:d_i=q\}$.
Since $\HH^q(F,R)$ is finite free, we obtain
\begin{equation}\label{eq:E2}
\begin{aligned}
 E_2^{p,q}
 &\cong \HH^p\bigl(B;\HH^q(F,R)\bigr)\\
 &\cong \bigoplus_{i\in I_q}\HH^p(B,R)\\
 &\cong \HH^p(B,R)\otimes_R\HH^q(F,R).
\end{aligned}
\end{equation}

\forget{
The spectral sequence is multiplicative, and its
differentials have bidegree $(s,1-s)$ and satisfy
the Leibniz rule
\[
 d_s(xy)=d_s(x)y+(-1)^{p+q}x\,d_s(y),
 \qquad x\in E_s^{p,q}.
\]
}

The restriction homomorphism is the fibre edge map
\begin{equation}\label{eq:edge}
 \iota^*:
 \HH^q(E,R)
 \twoheadrightarrow E_\infty^{0,q}
 \hookrightarrow E_2^{0,q}=\HH^q(F,R).
\end{equation}
Since no differential enters the leftmost column,
surjectivity of $\iota^*$ implies that every element
of $E_2^{0,q}$ is a permanent cycle.

On the other hand, every bottom-row class has zero
outgoing differential, since
\[
 d_s^{p,0}:E_s^{p,0}
 \longrightarrow E_s^{p+s,1-s}=0
 \qquad(s\geq2).
\]
By \eqref{eq:E2}, every element of $E_2^{p,q}$ is
a sum of products \((\alpha_i\otimes1)(1\otimes e_i),\) with \(\alpha_i\in\HH^p(B,R),\) for \(i\in I_q\).

The Leibniz rule therefore gives $d_2=0$. Repeating the same
argument inductively on each subsequent page yields
$d_s=0$ for every $s\geq2$. 
Consequently, the spectral
sequence degenerates at $E_2$, yielding
\[
 E_\infty^{p,q}
 \cong E_2^{p,q}
 \cong \HH^p(B,R)\otimes_R\HH^q(F,R).
\]
 As we explain below in
Remark~\ref{rem:leray-hirsch-map}, this
homomorphism preserves the natural filtrations.
\forget{and induces the isomorphism
\[
 E_2^{p,q}\xrightarrow{\sim}E_\infty^{p,q}
\]
on the associated graded modules. Since these
filtrations are finite in each degree, it is an
isomorphism. This gives the isomorphism asserted
in the theorem.}
\end{proof}

\begin{remark}\label{rem:leray-hirsch-map}
We describe explicitly the isomorphism constructed
in the proof of Theorem~\ref{thm:leray-hirsch}.
Notice that, choosing homogeneous lifts \(c_i\in\HH^{d_i}(E,R)
\) with \(\iota^*c_i=e_i,\) of the basis elements chosen above,  since $\HH^*(B,R)\otimes_R\HH^*(F,R)$ is a
free graded $\HH^*(B,R)$-module with homogeneous
basis $\{1\otimes e_i\}_{i\in I}$, the assignment \(1\otimes e_i\longmapsto c_i\) extends uniquely to a graded $\HH^*(B,R)$-linear
homomorphism. Since the action of $\HH^*(B,R)$ on $\HH^*(E,R)$
is induced by $\pi^*$, the resulting map is
necessarily given by
\[
\begin{aligned}
 \Phi:\HH^*(B,R)\otimes_R\HH^*(F,R)
 &\longrightarrow\HH^*(E,R),\\
 \alpha\otimes e_i
 &\longmapsto\pi^*\alpha\smile c_i.
\end{aligned}
\]
Thus $\Phi$ is precisely the homomorphism
constructed in the proof. We verify that $\Phi$ is
a filtered isomorphism. In degree $n$, equip \(
A^n=\bigoplus_{a+b=n}\HH^a(B,R)\otimes_R\HH^b(F,R)
\) by
\[
\mathcal G^pA^n=
\bigoplus_{\substack{a+b=n\\a\geq p}}
\HH^a(B,R)\otimes_R\HH^b(F,R).
\]
Let $\mathcal F^\bullet$ be the Serre filtration
on $\HH^n(E,R)$. Since
$\pi^*\HH^a(B,R)\subseteq\mathcal F^a\HH^a(E,R)$,
multiplicativity implies that $\Phi$ is filtered.
Moreover,
\[
\frac{\mathcal G^pA^n}{\mathcal G^{p+1}A^n}
\cong
\HH^p(B,R)\otimes_R\HH^{n-p}(F,R)
\cong
\frac{\mathcal F^p\HH^n(E,R)}
{\mathcal F^{p+1}\HH^n(E,R)}.
\]
By the edge maps and degeneration at $E_2$,
$\operatorname{gr}^p(\Phi)$ is the identity under
these identifications. 
Equivalently, in degree $n$, we have
\[
\bigoplus_{i\in I}\HH^{n-d_i}(B,R)
\xrightarrow{\ \sim\ }\HH^n(E,R),
\qquad
(\alpha_i)_i\longmapsto
\sum_i\pi^*\alpha_i\smile c_i.
\]
This isomorphism generally depends on the chosen lifts $c_i$.

\end{remark}

\section{Leray--Hirsch Theorem for Pure Tate fibre bundles}

Throughout, $\mathrm{CH}_m(S)$ denotes the homological Chow
group of dimension-$m$ cycles on $S$, whereas
$\mathrm{CH}_{\mathrm{op}}^j(S)$ denotes operational Chow
cohomology of degree $j$.

An operational class $c\in\mathrm{CH}_{\mathrm{op}}^j(S)$
acts, for every morphism $g:T\to S$, by
\[
  g^*c\cap-:
  \mathrm{CH}_m(T)\longrightarrow\mathrm{CH}_{m-j}(T).
\]
These operations commute with proper push-forward and flat
pull-back; see \cite[Chapter~17]{Ful}. If $S$ is smooth,
operational Chow cohomology gets identified with the ordinary
Chow ring indexed by codimension
\cite[Corollary~17.4]{Ful}.

\begin{theorem}\label{ThmLetaHirschForChow}
Let $Y$ be a smooth integral variety over a field $k$, with
function field $K=k(Y)$. Let $\pi:X\to Y$ be a smooth proper
morphism of relative dimension $d$, and assume that its
generic fibre $F=X_K$ is projective and has an integral
pure Tate Chow motive $M_{\mathrm{CH}}(F)=M_K$ with $M\simeq
  \bigoplus_{j=0}^d
       \bigl(\mathbb L^{\otimes j}\bigr)^{\oplus r_j}$.\forget{
------------------------
       
\[
  M:=M_{\mathrm{CH}}(F)
  \simeq
  \bigoplus_{j=0}^d
       \bigl(\mathbb L_K^{\otimes j}\bigr)^{\oplus r_j}.
\]

Equivalently, this decomposition is the extension to $K$
of the pure Tate motive
\[
  M=
  \bigoplus_{j=0}^d
       \bigl(\mathbb L_k^{\otimes j}\bigr)^{\oplus r_j}
\]
over $k$.
-------------------------
}
Then there exists an isomorphism of graded abelian groups
\[
  \mathrm{CH}_*(Y)\otimes_{\mathbb Z}\mathrm{CH}^*(F)
  \xrightarrow{\ \sim\ }\mathrm{CH}_*(X).
\]
where the grading on the left is defined by
\[
  \bigl(
    \mathrm{CH}_*(Y)\otimes_{\mathbb Z}\mathrm{CH}^*(F)
  \bigr)_m
  :=
  \bigoplus_{j=0}^{d}
    \mathrm{CH}_{m-d+j}(Y)
    \otimes_{\mathbb Z}\mathrm{CH}^j(F).
\]
More explicitly, for every integer $m$ there is an
isomorphism
\[
  \bigoplus_{j=0}^d
     \mathrm{CH}_{m-d+j}(Y)^{\oplus r_j}
  \xrightarrow{\ \sim\ }\mathrm{CH}_m(X).
\]
\end{theorem}

\begin{proof}
We first construct certain global classes and show that their
restrictions form bases of the Chow groups of the  fibres. The pure Tate decomposition gives homogeneous classes
\[
  e_i^j\in\mathrm{CH}^j(F),
  \qquad
  f_i^j\in\mathrm{CH}^{d-j}(F),
  \qquad 1\leq i\leq r_j,
\]
satisfying
\[
  \deg(f_q^j\cdot e_i^j)=\delta_{qi}
\]
and
\[
  [\Delta_F]
  =
  \sum_{j=0}^d\sum_{i=1}^{r_j}
      f_i^j\times e_i^j
  \quad\text{in }\mathrm{CH}^d(F\times_K F).
\]
Here the $e_i^j$ form a basis of $\mathrm{CH}^j(F)$.
The above identity follows by writing the identity
endomorphism of the pure Tate motive as the sum of its
Tate projectors.

Choose cycle representatives of the $e_i^j$ and $f_i^j$, and
spread these cycles over a nonempty open subset of $Y$
and take their closures in $X$. This produces classes
\[
  \zeta_i^j\in\mathrm{CH}^j(X),
  \qquad
  \vartheta_i^j\in\mathrm{CH}^{d-j}(X)
\]
whose generic restrictions are $e_i^j$ and $f_i^j$,
respectively.

Since $Y$ and $\pi$ are smooth, $X$ is smooth. We therefore
regard these classes as operational classes:
\[
  \zeta_i^j\in\mathrm{CH}_{\mathrm{op}}^j(X),
  \qquad
  \vartheta_i^j\in
       \mathrm{CH}_{\mathrm{op}}^{d-j}(X).
\]

Note that these classes pull back along every base change, including
those over singular closed subsets of $Y$; See first paragraph of Section \ref{SectNotation}.

\medskip
\noindent
For $y\in Y$, choose regular parameters $t_1,\ldots,t_c$
in $A=\mathcal O_{Y,y}$, here $c=\dim A$. The DVRs
\[
 \bigl(A/(t_1,\ldots,t_\ell)\bigr)_{(\overline t_{\ell+1})},
 \qquad 0\leq\ell<c,
\]
give successive specializations from $K$ to $k(y)$.
Chow specialization commutes with proper push-forward and preserves intersection product; See \cite[Section 20.3]{Ful}. Consequently, writing $\zeta_{i,y}^j$ and
$\vartheta_{i,y}^j$ for the restrictions to $X_y$, we obtain
\[
  \deg(\vartheta_{q,y}^j\cdot\zeta_{i,y}^j)
  =\delta_{qi}
\]
and
\[
  [\Delta_{X_y}]
  =
  \sum_{j=0}^d\sum_{i=1}^{r_j}
     \vartheta_{i,y}^j\times\zeta_{i,y}^j.
\]
Acting with this diagonal identity on
$\gamma\in\mathrm{CH}^q(X_y)$ gives
\[
  \gamma
  =
  \sum_{i=1}^{r_q}
     \deg(\gamma\cdot\vartheta_{i,y}^q)\,
     \zeta_{i,y}^q.
\]
\forget{Indeed, terms of other degrees contribute zero for
dimension reasons.} Thus the $\zeta_{i,y}^q$ generate
$\mathrm{CH}^q(X_y)$, and the degree relations show that
they are linearly independent. Thus we observe that for every point $y$,
the set \(\{\zeta_{i,y}^j\}_{i=1}^{r_j}\) is a $\mathbb Z$-basis of $\mathrm{CH}^j(X_y)$.

\medskip
\noindent
For every locally closed reduced subscheme $S\subseteq Y$,
write
\[
  X_S=X\times_Y S,\qquad \pi_S:X_S\to S,
\]
and let $\zeta_{i,S}^j$ denote the operational restriction
of $\zeta_i^j$ to $X_S$. Put
\[
  D_m(S)=
  \bigoplus_{j=0}^d\bigoplus_{i=1}^{r_j}
     \mathrm{CH}_{m-d+j}(S).
\]
Define
\[
  \phi_{S,m}:D_m(S)\longrightarrow\mathrm{CH}_m(X_S),
  \qquad
  (\alpha_i^j)\longmapsto
  \sum_{j,i}\zeta_{i,S}^j\cap\pi_S^*(\alpha_i^j).
\]

\medskip
\noindent

Now, we prove that $\phi_{S,m}$ is \emph{surjective} for every reduced
locally closed $S\subseteq Y$, by induction on $\dim S$.
Proper push-forward
from irreducible components reduces the proof to integral $S$. The surjectivity is obvious for $\dim S=0$.


Let $e=\dim S$, let $L=k(S)$, and take
$\xi\in\mathrm{CH}_m(X_S)$. Its generic restriction belongs to $\mathrm{CH}_{m-e}(X_L)$. The fibrewise basis property
proved above gives
\[
  \xi_L=
  \sum_{\substack{j,i\\d-j=m-e}}
    c_i^j\bigl(\zeta_{i,L}^j\cap[X_L]\bigr),
  \qquad c_i^j\in\mathbb Z.
\]
Set $\alpha_i^j=c_i^j[S]$ for these indices, and set the
remaining coefficients equal to zero. 

Then \(\delta:=\xi-\phi_{S,m}(\alpha)\) has zero generic restriction. Thus
$\delta$ restricts to zero over some nonempty open subset
$U\subseteq S$. Put $Z=(S\setminus U)_{\mathrm{red}}$.
Let
\[
  i:Z\hookrightarrow S,\qquad
  i':X_Z\hookrightarrow X_S
\]
be the closed immersions.

Localization gives a diagram with exact rows:
\[
\begin{CD}
D_m(Z) @>{i_*}>> D_m(S) @>>> D_m(U) @>>> 0\\
@V{\phi_{Z,m}}VV @V{\phi_{S,m}}VV @V{\phi_{U,m}}VV\\
\mathrm{CH}_m(X_Z) @>{i'_*}>>
\mathrm{CH}_m(X_S) @>>>
\mathrm{CH}_m(X_U) @>>> 0.
\end{CD}
\]
Note that the left vertical map is
\[
  \phi_{Z,m}((\beta_i^j))
  =
  \sum_{j,i}\zeta_{i,Z}^j\cap\pi_Z^*(\beta_i^j).
\]
It is defined using homological Chow groups and operational
classes, even though $Z$ and $X_Z$ may be singular. Moreover, the left square commutes because proper--flat base change
and the operational projection formula give
\[
\begin{aligned}
  \zeta_{i,S}^j\cap\pi_S^*(i_*\beta)
  &=
  \zeta_{i,S}^j\cap i'_*(\pi_Z^*\beta)\\
  &=
  i'_*\bigl(\zeta_{i,Z}^j\cap\pi_Z^*\beta\bigr).
\end{aligned}
\]

Localization provides $\eta\in\mathrm{CH}_m(X_Z)$ such
that $i'_*\eta=\delta$. By induction,
$\eta=\phi_{Z,m}(\beta)$ for some $\beta\in D_m(Z)$.
Hence
\[
  \xi
  =\phi_{S,m}(\alpha)+i'_*\phi_{Z,m}(\beta)
  =\phi_{S,m}(\alpha+i_*\beta).
\]
This proves surjectivity, in particular for $S=Y$.

\medskip
\noindent
Now we prove the \emph{Injectivity.} For the global classes constructed above, proper
push-forward gives \(\pi_*(\vartheta_q^p\cdot\zeta_i^j)
  \in\mathrm{CH}^{j-p}(Y).\)
This class vanishes when $j<p$. When $j=p$, it belongs
to $\mathrm{CH}^0(Y)=\mathbb Z[Y]$, and its coefficient
is determined by its generic restriction. The generic
duality relations therefore give
\[
  \pi_*(\vartheta_q^p\cdot\zeta_i^p)
  =\delta_{qi}[Y].
\]
By the projection formula, for every homological class
$\alpha$ on $Y$,
\[
  \pi_*\bigl(
    (\vartheta_q^p\zeta_i^j)\cap\pi^*\alpha
  \bigr)
  =
  \begin{cases}
    0,&j<p,\\
    \delta_{qi}\alpha,&j=p.
  \end{cases}
\]
Suppose that
\[
  \sum_{j,i}\zeta_i^j\cap\pi^*(\alpha_i^j)=0.
\]
If some coefficient is nonzero, choose the largest $p$
such that $\alpha_q^p\neq0$ for some $q$. Applying
$\vartheta_q^p\cap-$ and then $\pi_*$ gives
\[
  0
  =
  \sum_{j,i}
    \pi_*\bigl(
      (\vartheta_q^p\zeta_i^j)\cap\pi^*\alpha_i^j
    \bigr)
  =
  \alpha_q^p.
\]
The terms with $j>p$ have zero coefficients, those with
$j<p$ vanish, and those with $j=p$ select $\alpha_q^p$.
This contradiction proves injectivity.

Finally, identifying the summands of $D_m(Y)$ with the
chosen homogeneous basis $\{e_i^j\}$ of
$\mathrm{CH}^*(F)$ gives the asserted graded isomorphism.

\forget{
===============================
It remains to show $\ker \phi = 0$. Let $\tilde{\eta} \in \mathrm{CH}^d(F)$ be the generator corresponding to $1 \in \mathbb{Z}$. Relabel the homogeneous elements $\zeta_i$ which lie in $\mathrm{CH}^j(X)$ by double subscripts $\zeta_i^j$. Since $F$ satisfies Poincaré duality, we may choose elements $\vartheta_i^j \in \mathrm{CH}^{d-j}(X)$ whose restrictions to fibres give the dual basis\forget{ of the restrictions of $\zeta_i^j$}.

Now if
\[
  \sum_{i,j} \zeta_i^j \cap \pi^*\alpha_i^j = 0,
\]
then
\[
  \pi_*\bigl(\vartheta_q^p \cap (\sum_{i,j} \zeta_i^j \cap \pi^*\alpha_i^j)\bigr) = 0.
\]
Let $p$ be the maximal index for which $\alpha_q^p \ne 0$ for some $q$, therefore it suffices to check the cases where $j \leq p$. When $ i \ne q$, the duality gives $\pi_*(\vartheta_q^p \cap \zeta_i^p \cap \pi^*\alpha_i^p)=0$. Note in addition that $\pi_*(\vartheta_q^p \cap \zeta_q^p \cap \pi^*\alpha_q^p)=\alpha_q^p$. This follows from the following $\textit{general fact}$: 

--Let $[pt] \in \mathrm{CH}^d(F)$ be the generator corresponding to $1 \in \mathbb{Z} $ under the degree isomorphism. For an element $\gamma$ of $\mathrm{CH}^d(X)$, whose restriction to a particular fibre is $n[pt]$, for some integer $n$, we have $\pi_*(\gamma \cap \pi^*(\alpha)) =n \alpha$, for every $\alpha \in \mathrm{CH}_*(Y)$. \\
Similarly one can show that $\pi_*(\gamma \cap \pi^*(\alpha)) = 0$ for $\gamma \in \mathrm{CH}^s(X)$, with $s< d$. So for $j< p$ we have $\pi_*(\vartheta_q^p \cap \zeta_i^j \cap \pi^*\alpha_i^j) = 0$.

==================================
}

\end{proof}

\begin{remark}
Note that the assumption that the fibre $F$ is pure Tate is crucial even to have K\"unneth formula. E.g. for abelian variety $A$, one can see that the class of Poincar\'e line bundle in $\mathrm{CH}^1(A\times A^\vee)$ is not coming from $\mathrm{CH}^*(A)\otimes \mathrm{CH}^*(A^\vee) \to \mathrm{CH}^*(A\times A^\vee)$. 
\end{remark}

\begin{theorem}\label{thm:chowmotive}
Under Theorem~\ref{ThmLetaHirschForChow}, suppose in addition that $Y$ is projective. Then there is an isomorphism
\[
 M_{\mathrm{CH}}(Y)\otimes M\xrightarrow{\sim}M_{\mathrm{CH}}(X)
 \quad\text{in }\CHM(k).
\]

\end{theorem}

\begin{proof}
Put $s=\dim Y$, and retain the classes
\[
  \zeta_i^j\in\mathrm{CH}^j(X),
  \qquad
  0\leq j\leq d,\quad 1\leq i\leq r_j,
\]
constructed in the proof of
Theorem~\ref{ThmLetaHirschForChow}.
Let $\Gamma_\pi^t\subset Y\times X$ be the transposed
graph of $\pi$, and let $p_X:Y\times X\to X$ be the
projection. Define
\[
  C_i^j
  :=
  p_X^*\zeta_i^j\cap[\Gamma_\pi^t]
  \in
  \mathrm{CH}_{s+d-j}(Y\times X).
\]
With our covariant convention, this correspondence
defines a morphism
\[
  C_i^j:
  M_{\mathrm{CH}}(Y)\otimes
  \mathbb L^{\otimes(d-j)}
  \longrightarrow M_{\mathrm{CH}}(X).
\]
Taking their sum gives
\begin{equation}\label{EqLeray-HirschForChowMotives}
  \bigoplus_{j=0}^d\bigoplus_{i=1}^{r_j}
  M_{\mathrm{CH}}(Y)\otimes
  \mathbb L^{\otimes(d-j)}
  \longrightarrow M_{\mathrm{CH}}(X).
\end{equation}

For every smooth projective equidimensional $k$-scheme $T$
and every integer $n$, applying
$\Hom(M_{\mathrm{CH}}(T)\otimes\LL^n,-)$ yields an
isomorphism by Theorem~\ref{ThmLetaHirschForChow}
applied to $X_T\to Y_T$. Hence the above morphism
is an isomorphism by Manin's identity principle;
see \cite{Manin}.

Note that the perfect intersection pairing on the
smooth projective pure Tate generic fibre gives
$r_j=r_{d-j}$. \forget{Therefore
\[
  \bigoplus_{j=0}^d
  \left(\mathbb L^{\otimes(d-j)}\right)^{\oplus r_j}
  \simeq
  \bigoplus_{j=0}^d
  \left(\mathbb L^{\otimes j}\right)^{\oplus r_j}
  =M.
\]}
Identifying the source of \eqref{EqLeray-HirschForChowMotives} accordingly yields
\[
  M_{\mathrm{CH}}(Y)\otimes M
  \xrightarrow{\ \sim\ }
  M_{\mathrm{CH}}(X),
\]
as asserted.
\end{proof}

\section{Relative motives and compact support}

We now assume the coefficients $R$, provided that the exponential characteristic of the base field is inverted; see Section \ref{SectNotation}. 

\begin{theorem}\label{thm:motivic}
Let $\pi:X\to Y$ satisfy Theorem~\ref{ThmLetaHirschForChow} with coefficients in $R$. Then there is an isomorphism
\begin{equation}\label{eq:relative}
 \theta:\bigoplus_{i=0}^d\one_Y(-i)[-2i]^{\oplus r_i}
 \xrightarrow{\sim}\pi_*\one_X
 \quad\text{in }\DM(Y,R).
\end{equation}
In particular, identifying the realization of $M$ with
$\bigoplus_i R(i)[2i]^{\oplus r_i}$, we obtain
\[
 \Mcomp(Y)\otimes M\xrightarrow{\sim}\Mcomp(X).
\]
\end{theorem}

\begin{proof}
	Since $X$ is smooth over $k$, each class
	$\zeta_i^j \in\mathrm{CH}^j(X;R)$ (see proof of Theorem \ref{ThmLetaHirschForChow}), corresponds, under the
	identification with motivic cohomology, to a morphism
	\[
	\one_X\longrightarrow\one_X(j)[2j];
	\]
	See \eqref{eq:higherhom}. Adjunction and the projection formula give
	\[
	\Hom_{\DM(X,R)}(\one_X,\one_X(j)[2j])
	\simeq
	\Hom_{\DM(Y,R)}
	\bigl(\one_Y(-j)[-2j],\pi_*\one_X\bigr).
	\]
	Taking the sum over $i$ and $j$ defines
	\[
	\theta:
	\bigoplus_{j=0}^d
	\one_Y(-j)[-2j]^{\oplus r_j}
	\longrightarrow \pi_*\one_X.
	\]
	For each scheme point $y:\Spec k(y)\to Y$, proper base
	change identifies $y^*\theta$ with
	\[
	\bigoplus_{j=0}^d
	\one_{k(y)}(-j)[-2j]^{\oplus r_j}
	\longrightarrow (\pi_y)_*\one_{X_y},
	\]
	represented by the restrictions $\zeta_i^j|_{X_y}$.
	Since $X_y$ is smooth and proper, the target is the
	dual of $M(X_y)$. The dual of $y^*\theta$ is therefore
	the morphism
	\[
	M(X_y)\longrightarrow
	\bigoplus_{j=0}^d R(j)[2j]^{\oplus r_j}.
	\]
    Hence $y^*\theta$ is an	isomorphism at every scheme point $y$ of $Y$.
	Let $C$ be the cone of $\theta$. It is constructible
	and vanishes at every scheme point of $Y$. By continuity,
	the vanishing at the generic point implies that the
	identity of $C$ vanishes on some dense open $U\subset Y$,
	so $C|_U=0$.
	
	Let $j:U\hookrightarrow Y$ be the open immersion, and
	$\iota:Z\hookrightarrow Y$ the complementary closed immersion.
	By induction on dimension, $\iota^*C=0$: after passing
	to the reduction, the same continuity argument applies
	on each irreducible component of $Z$. The localization
	triangle
	\[
	j_!j^*C\longrightarrow C\longrightarrow
	\iota_*\iota^*C\longrightarrow j_!j^*C[1]
	\]
	then gives $C=0$, proving \eqref{eq:relative}; see \cite[Sections 2, 5]{CD}.
	
	Finally, let $a:Y\to\Spec k$ and $b:X\to\Spec k$
	be the structure morphisms. Since $\pi$ is proper,
	$\pi_*=\pi_!$. Applying $a_!$ to $\theta$ and using the
	projection formula gives
	\[
	\bigoplus_{i=0}^d
	a_!\one_Y(-i)[-2i]^{\oplus r_i}
	\xrightarrow{\sim}b_!\one_X.
	\]
	Dualizing over $k$ and using \eqref{eq:mcdual}, we obtain
	\[
	\Mcomp(X)\xrightarrow{\sim}
	\bigoplus_{i=0}^d
	\Mcomp(Y)(i)[2i]^{\oplus r_i}
	\simeq \Mcomp(Y)\otimes M.
	\]
    Full faithfulness of stabilization
	gives this isomorphism in the effective category.
\end{proof}

\section{Consequences}

The following Corollary provides Leray--Hirsch splitting for higher Chow groups:

\begin{corollary}\label{CorHigherChow}
Let $\pi:X\to Y$ and $M$ be as in
Theorem~\ref{thm:motivic}\forget{, and assume that $X$ and $Y$
are quasi-projective. Let $R\subseteq\mathbb Q$ be a
subring in which the exponential characteristic of $k$
is invertible.} For every $p\in\mathbb Z$ and
$q\geq0$, there is an isomorphism
\[
  \mathrm{CH}^p(X,q;R)
  \simeq
  \bigoplus_{i=0}^{d}
  \left(
    \mathrm{CH}^{p-i}(Y,q;R)
    \otimes_R\mathrm{CH}^i(F;R)
  \right),
  \qquad d=\dim F.
\]
\end{corollary}

\begin{proof}
Let $M$ denote the geometric realization over $k$
of the pure Tate motive whose base change to $K$
is $M_{\mathrm{CH}}(F)$.
Theorem~\ref{thm:motivic} gives
\[
  M^c(X)\simeq M^c(Y)\otimes_R M.
\]
By Proposition~\ref{PropCHowDecompositionForPM}\forget{,
using $\mathrm{CH}_{d-i}(F;R)=\mathrm{CH}^i(F;R)$},
we have
\[
  M_R\simeq
  \bigoplus_{i=0}^{d}
  \mathrm{CH}^i(F;R)\otimes_R
  R(d-i)[2(d-i)].
\]
Consequently,
\[
  M^c(X)\simeq
  \bigoplus_{i=0}^{d}
  \mathrm{CH}^i(F)\otimes_R
  M^c(Y)(d-i)[2(d-i)].
\]

Put $s=\dim Y$, so that $\dim X=s+d$.
The motivic description of higher Chow groups \eqref{eq:higherhom1} gives
\[
\begin{aligned}
  \mathrm{CH}^p(X,q;R)
  &\simeq
  \operatorname{Hom}_{\mathbf{DM}(k,R)}
  \bigl(
    R(s+d-p)[2(s+d-p)+q],M^c(X;R)
  \bigr)\\
  &\simeq
  \bigoplus_{i=0}^{d}
  \mathrm{CH}^i(F;R)\otimes_R
  \operatorname{Hom}_{\mathbf{DM}(k,R)}
  \bigl(
    R(s-p+i)[2(s-p+i)+q],M^c(Y;R)
  \bigr)\\
  &\simeq
  \bigoplus_{i=0}^{d}
  \mathrm{CH}^{p-i}(Y,q;R)
  \otimes_R\mathrm{CH}^i(F;R).
\end{aligned}
\]
In the second step, we used that the groups
$\mathrm{CH}^i(F;R)$ are finite free $R$-modules.
\end{proof}

\noindent
\textbf{Projective homogeneous bundles for reductive groups.} Let \(G\) be a split reductive algebraic group over $k$ and \(P\subset G\) a parabolic subgroup. The homogeneous space \(G/P\) is projective and admits a Bruhat (Schubert) cell decomposition.
   For a principal \(G\)-bundle \(\mathcal{P}\to Y\) over smooth projective base $Y$, form the associated bundle \(X=\mathcal{P}\times^G (G/P)\). Assume further that the class of $[\mathcal P_\eta]$ is trivial in $\mathrm H^1(k(Y),G)$.
   Since the fibre \(G/P\) is cellular, and hence pure Tate, and moreover $[\mathcal P_\eta]$ is trivial,  Schubert cycles give the global classes on \(X\). Thus the 
   Leray--Hirsch theorem gives
  \[
    M(X)\ \cong\ \bigoplus_{w\in W^P} M(Y)\big(\ell(w)\big)[2\ell(w)\big],
  \]
  where \(W^P\) are minimal coset representatives of the Weyl group and \(\ell(w)\) is the length of \(w\).

\paragraph{The Case of $G=GL_r$}
Let \(\mathcal{E}\) be a rank-\(r\) vector bundle on \(Y\). Let \(\pi:\mathrm{Gr}_d(\mathcal{E})\to Y\) be the associated Grassmann bundle of rank-\(d\) subbundles.
  The fibre is the Grassmannian \(\mathrm{Gr}(d,r)\), which is cellular. Since the family satisfies the assumptions in the theorem \ref{thm:chowmotive}, the theorem yields
  \[
    M\big(\mathrm{Gr}_d(\mathcal{E})\big)\ \cong\ \bigoplus_{\lambda} M(Y)(|\lambda|)[2|\lambda|],
  \]
  where \(\lambda\) runs over partitions indexing the Schubert cells and \(|\lambda|\) denotes the weight. Note that for \(d=1\) this recovers the projective bundle formula  \[
    M\big(\mathbb{P}(\mathcal{E})\big) \;\cong\; \bigoplus_{i=0}^{r-1} M(Y)(i)[2i].
  \]

\begin{remark}[Chow--K\"unneth decompositions]
 The theorem provides a Chow--K\"unneth decomposition for \(X\) whenever one exists for \(Y\).
   If \(M(Y)\) admits a Chow--K\"unneth decomposition, for instance, for abelian varieties, see \cite{D-M}, then so does \(X\) by tensoring with the pure Tate motive of the fibre. 
\end{remark}

\begin{remark}
Another important application  comes from the Motivic Decomposition Theorem (MDT) for semi-small resolutions, established by Migliorini and de Cataldo \cite{C-M}. Namely, when a variety $S$ admits a semi-small resolution of singularities $\Sigma$, which can be realized as an iterated tower of fibrations, satisfying the assumption in the Theorem \ref{thm:chowmotive}, with fibre $F_i$, then one can compute the motive of $\Sigma$ in terms of the fibres, and then use MDT to relate it to the motive $M(S)$ (with rational coefficients). This is useful to study the motives associated to Schubert varieties inside affine Grassmannian $Gr(G)=G(k((z)))/G(k\llbracket z\rrbracket)$, associated with a split reductive group $G$. This in particular implies that the Chow groups of a Schubert variety $S(\mu)$ in affine Grassmannian $Gr(G)$ are free and finitely generated. For example, see the situation considered by Ng\^o and Polo \cite{NgoPolo}[Lemma 9.3], where $\mu$ is a sum of cocharacters $\mu_i$ and the resolution is semi-small. 
   
\end{remark}

\begin{remark}
 Motivic decompositions frequently mirror decompositions in derived categories of coherent sheaves.
  Projective-bundle and Grassmann-bundle formulas for motives reflect the semiorthogonal decompositions of the corresponding derived categories; e.g. see \cite{Kap} and \cite{Orl}.
  This correspondence has been exploited in homological projective duality and related areas.

\end{remark}

\begin{remark}
An analogue version of Leray-Hirsch theorem for derived categories was proved by Costa, Di Rocco
and Mir\'o-Roig \cite[Theorem~2.8]{CDM} for a Zariski-locally
trivial fibration of smooth complex projective varieties, they
construct a full strongly exceptional collection of line bundles
on the total space from such collections on the base and fibre.
For further perspectives on derived categories, see Kapranov
\cite{Kap} for exceptional collections on certain homogeneous
spaces, see also \cite{Orl} and \cite{Orl1}.
\end{remark}

\section*{Acknowledgments} The second-named author is grateful to Luca Barbieri Viale and Bruno Kahn for their encouragement and steady support. Both authors are deeply grateful to the anonymous referee for numerous valuable and insightful comments, suggestions, and corrections that substantially improved the paper.

\begin{minipage}[t]{0.9\linewidth}
\noindent
\small\textbf{Esmail Arasteh Rad}, School of Mathematics, Institute for Research in Fundamental Sciences (IPM), P.O. Box: 19395-5746, Tehran, Iran 
email: earasteh@ipm.ir\\
\\[1mm]

\noindent
\small\textbf{Somayeh Habibi}, School of Mathematics, Institute for Research in Fundamental Sc
iences (IPM), P.O. Box: 19395-5746, Tehran, Iran
 email: \href{shabibi@ipm.ir}{shabibi@ipm.ir}

\end{minipage}

\end{document}